\documentclass{amsart}

\usepackage[all]{xy}

\usepackage{latexsym}

\usepackage{amssymb}
\usepackage{amsfonts}
\usepackage{amscd}
\usepackage{amsmath,amsthm}
\usepackage[dvipsnames]{xcolor}

\newtheorem{lemma}{Lemma}[section]
\newtheorem{proposition}[lemma]{Proposition}
\newtheorem{theorem}[lemma]{Theorem}
\newtheorem{corollary}[lemma]{Corollary}

\newtheorem{lemmas}{Lemma}[subsection]

\newtheorem{rem}[lemma]{Remark}

\newtheorem{prop}[lemma]{Proposition}

\newtheorem{cor}[lemma]{Corollary}

{

}

\theoremstyle{definition}

\newtheorem{definitions}[lemmas]{Definition}

\theoremstyle{remark}

\newcommand{\rto}{\rightarrow}
\newcommand{\lrto}{\longrightarrow}

\newcommand{\Hom}{\operatorname{Hom}}

\newcommand{\im}{\operatorname{im}}

\newcommand{\End}{\operatorname{End}}

\newcommand{\mbbS}{\mathbb{S}}

\numberwithin{equation}{section}

\begin{document}

\title{A representation theoretic study of Noncommutative Symmetric Algebras}

\author{D. Chan}
\address{University of New South Wales}
\email{danielc@unsw.edu.au}

\author{A. Nyman}
\address{Western Washington University}
\email{adam.nyman@wwu.edu}

\begin{abstract}
We study Van den Bergh's noncommutative symmetric algebra $\mbbS^{nc}(M)$ (over division rings) via Minamoto's theory of Fano algebras. In particular, we show  $\mbbS^{nc}(M)$ is coherent, and its proj category $\mathbb{P}^{nc}(M)$ is derived equivalent to the corresponding bimodule species. This generalizes the main theorem of \cite{minamoto}, which in turn is a generalization of Beilinson's derived equivalence. As corollaries, we show that $\mathbb{P}^{nc}(M)$ is hereditary and there is a structure theorem for sheaves on $\mathbb{P}^{nc}(M)$ analogous to that for $\mathbb{P}^1$.
\end{abstract}

\maketitle

\pagenumbering{arabic}

\section{Introduction}

The symmetric algebra $\mathbb{S}(V)$ on a finite dimensional vector space $V$ is a fundamental object in algebra that can be used to study the projective space $\mathbb{P}(V)$. Replacing the vector space $V$ with a fairly general finite bimodule over a pair of division rings (see Section \ref{section.ncsym} for precise conditions), one can form the noncommutative symmetric algebra $\mathbb{S}^{nc}(M)$ as defined by Van den Bergh \cite{vandenbergh}. When $M$ is 2-dimensional on the left and right, we studied the noncommutative symmetric algebra via classical techniques in noncommutative algebraic geometry in \cite{newchan}. In this case, its associated proj category $\mathbb{P}^{nc}(M)$ behaves much like $\mathbb{P}^1$. Indeed, $\mathbb{S}^{nc}(M)$ is noetherian and coherent sheaves on $\mathbb{P}^{nc}(M)$ are direct sums of their torsion part and line bundles.


In this note, we study the noncommutative symmetric algebra for higher dimensional $M$ extending the results of \cite{abstractp1}. The resulting algebra diverges sharply from the classical symmetric algebra and is in fact non-noetherian. For example, when $M$ is an $n$-dimensional vector space over a field $k$, then $\mathbb{S}^{nc}(M)$ is the $\mathbb{Z}$-indexed incarnation of the graded algebra $k[x_1,\ldots,x_n]/(\Sigma x_i^2)$ and its proj category behaves more like a projective line which Piontkovski dubs the $n$th projective line $\mathbb{P}^{1}_{n}$. Furthermore, it has been observed by Minamoto \cite{minamoto}, Piontkovski \cite{piont} and Van den Bergh that $\mathbb{P}^{1}_{n}$ is derived equivalent to the finite dimensional algebra
$\left( \begin{smallmatrix}
         k & M \\ 0 & k
        \end{smallmatrix}\right)$,
a result generalizing Beilinson's classic derived equivalence for $\mathbb{P}^1$. These results suggest that a more fruitful way to study noncommutative symmetric algebras is to first prove a version of Beilinson's derived equivalence in this context, and then extract desirable properties of $\mbbS^{nc}(M)$ as byproducts of the representation theory of (not necessarily finite-dimensional) bimodule species. The purpose of this note is to pursue this line of thought and hence show that $\mbbS^{nc}(M)$ is coherent, $\mathbb{P}^{nc}(M)$ is hereditary and there is a Grothendieck splitting theorem. This recovers many of the results of \cite{newchan}, in a more general context, by much simpler means. Thus, though the new representation-theoretic results here are quite modest, the implications for the noncommutative symmetric algebra are rather significant.

Finally, we remark that Van den Bergh's original motivation for introducing the noncommutative symmetric algebra was to study noncommutative ruled surfaces such as the 2-generator 3-dimensional Sklyanin algebras, where the most interesting cases occur when the corresponding bimodule species is not finite dimensional. We hope this paper will illuminate the study of noncommutative ruled surfaces.

\section{Noncommutative symmetric algebras} \label{section.ncsym}
Let $D_0$ and $D_1$ be division rings.  In this section, following \cite{vandenbergh}, we define the noncommutative symmetric algebra of certain $D_0-D_1$-bimodules.

\subsection{Bimodules}
Let $M$ be a $D_0-D_1$-bimodule. The {\it right dual of $M$}, denoted $M^{*}$, is the $D_1-D_0$-bimodule
$\operatorname{Hom}_{D_1}(M_{D_1},D_1)$, whilst the {\it left dual of $M$}, denoted ${}^{*}M$, is the $D_1-D_0$-bimodule $\operatorname{Hom}_{D_0}({}_{D_0}M,D_0)$



We need to iterate these duals and so introduce the following notation.
$$
M^{i*}:=
\begin{cases}
M & \text{if $i=0$}, \\
(M^{i-1*})^{*} & \text{ if $i>0$}, \\
{}^{*}(M^{i+1*}) & \text{ if $i<0$}.
\end{cases}
$$
As in \cite{dlab}, we need to impose a condition on the bimodule to ensure it is well behaved (see Section~\ref{sec:species} for why this is so).
\begin{definitions}
We say that $M$ has {\em symmetric duals} if $M, M^*$ are finite dimensional on the left and right, and there is a bimodule isomorphism $M \cong M^{**}$.
\end{definitions}
In this case, all the $M^{i*}$ are finite dimensional on both sides and ${}^*M \simeq M^*$, hence the terminology.
If $M$ has finite left-dimension $m$ and finite right-dimension $n$, we say $M$ has {\it left-right dimension }$(m,n)$.  The next proposition gives some instances of when bimodules have symmetric duals.

\begin{prop}
Suppose $M$ has left-right-dimension $(m,n)$.  Then $M$ has symmetric duals if
\begin{enumerate}
\item{} $D_{0}$ and $D_{1}$ are finite-dimensional over $k$ and $\operatorname{char }k$ does not divide either $[D_{0}:k]$ or $[D_{1}:k]$,

\item{} $D_{1}$ is a commutative subring of $D_{0}$ such that $[D_{0}:D_{1}]=m < \infty$, $\operatorname{char }k$ does not divide $m$, and $M={}_{D_{0}}{D_{0}}_{D_{1}}$, or

\item{} $D_{0}$ and $D_{1}$ are commutative, $M$ is simple of left- and right-dimension $(m,n)$, and the characteristic of $k$ does not divide $m$ or $n$.
\end{enumerate}
\end{prop}

\begin{proof}
To prove the first result (which appeared in \cite{olddlab}), one shows that there are $D_{1}-D_{0}$-bimodule isomorphisms
$$
\operatorname{Hom}_{D_{1}}(M_{D_{1}},D_{1}) \rightarrow \operatorname{Hom}_{k}(M,k)
$$
and
$$
\operatorname{Hom}_{D_{0}}({}_{D_{0}}M, D_{0}) \rightarrow \operatorname{Hom}_{k}(M,k).
$$
The first one takes $\psi:M_{D_{1}} \rightarrow D_{1}$ to $\operatorname{tr}_{D_{1}/k} \circ \psi$, and the second is similar.

The proofs of the second and third results follow the proof of \cite[Lemma 3.2]{newchan}.
\end{proof}



\subsection{The definition of $\mathbb{S}^{nc}(M)$}

For $i \in \mathbb{Z}$, we let $D_{i} = D_{\bar{i}}$ where $\bar{i}$ is the residue class of $i$ modulo 2. In what follows, all unadorned tensor products will be over $D_{i}$, the context determining uniquely which $i$ is required.

We fix a $D_0-D_1$-bimodule $M$ with symmetric duals and left-right dimension $(m,n)$ satisfying $mn \geq 4$. For each $i$, the following pairs of functors have canonical adjoint structures:
\begin{equation} \label{eqn.adjointone}
(-\otimes_{D_{i}} M^{i*},-\otimes_{D_{i+1}} M^{i+1*}).
\end{equation}
In particular, adjunction gives a natural map $\eta_i \colon D_{i} \rto M^{i*}\otimes_{D_{i+1}} M^{i+1*}$ whose image we denote by $Q_{i}$. If $\{\phi_{1}, \ldots, \phi_{n}\}$ is a right basis for $M^{i*}$ and $\{\phi_{1}^{*}, \ldots, \phi_{n}^{*}\}$ is a corresponding dual left basis for $M^{i+1*}$, then $\eta_i(1) = \sum_{i}\phi_{i} \otimes \phi_{i}^{*}$.  In particular, the latter element is $D_{i}$-central.  We will employ this fact without comment in the sequel.

We briefly recall Van den Bergh's definition of a noncommutative symmetric algebra, in the context we need. For further details, the interested reader should refer to the original paper \cite{vandenbergh}, or look at the gentler treatment in \cite[Section~3]{abstractp1}. The {\it noncommutative symmetric algebra of $M$}, denoted $\mathbb{S}^{nc}(M)$, is the positive $\mathbb{Z}$-indexed algebra $\mbbS = \underset{i,j \in \mathbb{Z}}{\oplus}\mbbS_{ij}$ defined via generators and relations as follows.
\begin{itemize}
\item In degree 0 we set $\mbbS_{ii}=D_i$,
\item $\mbbS$ is generated (over $\oplus \mbbS_{ii})$ in degree one by $\mbbS_{ii+1}=M^{i*}$ (our convention for multiplication is that $\mbbS_{ij}\mbbS_{jk} \subseteq \mbbS_{ik}$).
\item The relations are generated in degree two by $Q_i \subset M^{i*}\otimes_{D_{i+1}} M^{i+1*}$.
\end{itemize}

\begin{rem}  \label{rem.2periodic}
Since we are assuming $M \cong M^{**}$, we have an isomorphism of indexed algebras $\mbbS^{nc}(M) \cong \mbbS^{nc}(M^{**})$, and in particular, $\mbbS^{nc}(M)_{ij}  =  \mbbS^{nc}(M^{**})_{ij} = \mbbS^{nc}(M)_{i+2,j+2}$. We say, consequently, that $\mbbS^{nc}(M)$ is {\em 2-periodic}.
\end{rem}

In what follows, we will often write $\mathbb{S}$ instead of $\mathbb{S}^{nc}(M)$, and, where no confusion will arise, we will write $Q$ instead of $Q_{i}$.  Finally, we will let $\varepsilon_{i} \in \mathbb{S}_{ii}$ denote the unit.

The above definition for $\mbbS$ makes perfect sense even when $mn < 4$. However, in this case, $\mbbS$ degenerates and we no longer have Euler exact sequences as per Theorem \ref{thm.adam} (see \cite{abstractp1} for further details).




\section{Canonical complexes for artinian rings} \label{sec:species}
In this section, we look at an analogue of the Serre functor for artinian hereditary rings $A$ which are not necessarily finite-dimensional algebras. The non-derived versions have been studied briefly in \cite{apr} and \cite{dlab}. The vast majority of the literature however, assumes finite-dimensionality.

When $A$ is a finite-dimensional hereditary $k$-algebra, the $k$-linear dual of $A$ is an $A$-bimodule which is injective on the right (and left) and contains all the simple modules.
\begin{definitions} \label{def.canonical}
Suppose $DA$ is a right injective $A$-module such that i) there is an isomorphism $\End_A((DA)_{A}) \cong A$, and ii) $DA$ contains all the simple modules of $A$. Then we say the complex of $A$-bimodules $\omega = (DA)[-1]$ is a {\em canonical complex for $A$}, and that $A$ {\em has a canonical complex}.
\end{definitions}
Unfortunately, the bimodule structure of $DA$ depends on the choice of isomorphism $A \cong \End_A ((DA)_A)$.

For applications to the noncommutative symmetric algebra $\mbbS^{nc}(M)$ associated to the $D_{0}-D_{1}$-bimodule $M$ with symmetric duals, we need Ringel's bimodule species $A_M = \begin{pmatrix} D_{0} & M \\ 0 & D_{1} \end{pmatrix}$. The case where $M$ is an $n$-dimensional vector space over a field $k$ corresponds to the path algebra of the $n$-Kronecker quiver. We let $e_{0}$ and $e_{1}$ denote the diagonal idempotents of $A$ corresponding to $D_0, D_1$. We will usually write right $A_M$-modules $N$ as row vectors $N = (Ne_0 \ Ne_1)$. Now, by \cite[III Proposition 2.1]{aus}, $A_M$  is an artinian ring, which is not usually a finite dimensional algebra. Furthermore, the Jacobson radical of $A_M$ is
$$
\operatorname{rad}\, A_M =
\begin{pmatrix}
0 & M \\ 0 & 0
\end{pmatrix} \simeq
(0 \ \ D_1)^{\dim_{D_1} M}
$$
This is projective so \cite[I Corollary 5.2]{aus} ensures that $A_M$ is hereditary.

We  introduce the following $A_M$-bimodule $DA_M$:  as a group, $DA_M = \begin{pmatrix} D_{0} & 0 \\ M^{*} & D_{1} \end{pmatrix}$, with left-action defined by
$$
\begin{pmatrix} \alpha & m \\ 0 & \beta \end{pmatrix} \cdot \begin{pmatrix} a & 0 \\ \delta & b \end{pmatrix} := \begin{pmatrix} \alpha a + m(\delta) & 0 \\ \beta \delta & \beta b \end{pmatrix}
$$
and right-action defined by
$$
\begin{pmatrix} a & 0 \\ \delta & b \end{pmatrix} \cdot \begin{pmatrix} \alpha & m \\ 0 & \beta \end{pmatrix} := \begin{pmatrix} a \alpha & 0 \\ \delta \alpha & \delta(m)+ b \beta \end{pmatrix}
$$
where we have used the identification $M \cong M^{**}$ in our definition of the first action.

\begin{lemma}  \label{lem.canonical}
The bimodule $DA_M$ is the injective hull of the semisimple right $A$-module $(D_0 \ \ 0) \oplus ( 0 \ \ D_1)$ and $\omega = (DA_M)[-1]$ is a canonical complex for $A_M$.
\end{lemma}
\begin{proof}
Note that $DA_M$ is an $A$-bimodule so there is an induced morphism $A_M \rto \End (DA_M)_{A_M}$ which is easily checked to be an isomorphism. It thus suffices to show that the direct summands $(D_0 \ 0)$ and $(M^* \ D_1)$ are injective. This is clear in the former case so we check the latter, using Baer's criterion.  Since $e_{0}$ and $e_{1}$ are the diagonal idempotents of $A_{M}$, it suffices to show that if $N = (N_0 \ N_1)$ is a submodule of the projective module $e_i A_M$, where $i=0$ or $i=1$, and if $\phi\colon (N_0 \ N_1) \rightarrow (M^* \ D_1)$ is $A_M$-linear, then we can lift $\phi$ to $\phi' \colon e_iA_M \rightarrow (M^* \ D_1) $. Now $e_1 A_M = (0 \ D_1)$ is simple, so when $i=1$ we are done as $(N_0 \ N_1)$ is either 0 or all of $e_1 A_M$. Suppose now that $N \leq e_0 A_M$. We are done if $N = e_0 A_M$, so we may assume that $N_0 = 0$. Thus $\phi$ is given by a $D_1$-linear map $N_1 \rightarrow D_1$, which we can lift to a linear map $\phi' \in M^*$. This defines the required lift $(D_0 \ M) \rightarrow (M^* \ D_1)$.
\end{proof}

Returning to the general setup of a hereditary artinian ring $A$, we immediately have
\begin{proposition}  \label{prop:DAtilting}
Any canonical complex $\omega$ for $A$ is a tilting complex inducing an auto-equivalence of $D^b_{fg}(A)$. In particular, there is a complex $\omega^{-1}$ of bimodules, such that
$$  - \otimes^{L}_A\omega^{-1} = \operatorname{RHom}_A(\omega,-) $$
is inverse to $- \otimes^L_A \omega$.
\end{proposition}

We now assume that $A$ has a canonical complex $\omega$.
\begin{definitions} \label{def.regular}
A finitely generated $A$-module $N$ is said to be {\em regular} if
$$
N \otimes^L_A \omega^n \in {\sf mod}\, A
$$
for all $n \in \mathbb{Z}$.  We let {\sf R} denote the full subcategory of regular $A$-modules.
\end{definitions}

The following result is standard, but is invariably stated with a finite-dimensionality hypothesis, so we include the proof in order that the reader may easily check that the hypothesis may be relaxed.
\begin{lemma}  \label{lem.regular}
Let $A$ be an hereditary artinian algebra with a canonical complex $\omega = (DA)[-1]$.
\begin{enumerate}
\item If $N$ is a finitely generated indecomposable $A$-module such that $N \otimes^L_A \omega^{-1}$ is not a module, then $N$ is a direct summand of $DA$ in which case $N \otimes^L_A \omega^{-1} \cong P[1]$ for some projective module $P$.
\item Any finitely generated indecomposable $A$-module which is not regular has the form $I \otimes^L_A \omega^n$ for some injective module $I$ and $n \in \mathbb{N}$ or the form $P \otimes^L_A \omega^{-n}$ for some projective module $P$ and $n \in \mathbb{N}$.
\item $\Hom_{D^b_{fg}(A)}({\sf R} , \omega^n) = 0$ for all $n \in \mathbb{Z}$.
\end{enumerate}
\end{lemma}
\begin{proof}
We assume the hypotheses in part~(1) and recall that, since $A$ is hereditary, every indecomposable in $D^b_{fg}(A)$ has the form $L[j]$ for some indecomposable $A$-module $L$ and $j \in \mathbb{Z}$. Now $\omega$ lives in cohomological degree 1 and $N \otimes^L_A \omega^{-1}$ is indecomposable too, so the only possibility is that $N \otimes^L_A \omega^{-1} \cong P[1]$ where $P = \Hom_A(DA, N) \neq 0$. Picking any non-zero homomorphism $\phi\colon DA \rto N$, we see injectivity of $DA$ implies injectivity of $\im \phi$. Indecomposability of $N$ ensures that $N = \im \phi$. Now $DA$ contains all the simple modules, so the indecomposable injective module $N$ must be a direct summand of $DA$. We also see $P$ is projective since $A[1] \otimes^L_A \omega \cong DA$. This completes the proof of (1) from which part (2) readily follows.

Suppose now that $N$ is a regular module, so the same is true of $N \otimes^L_A \omega^{-n}$. Part~(3) thus follows if we can show that $\Hom_A(N, A) = 0$. If not, let $\phi\colon N \rto A$ be a non-zero homomorphism. Now $A$ is hereditary, so $P:= \im \phi$ is a non-zero projective summand of $N$. But $P \otimes^L_A \omega = P \otimes^L_A (DA)[-1]$ has non-zero cohomology in degree 1, contradicting regularity of $N$.
\end{proof}

\section{Preprojective and preinjective objects}
\label{sec:preproj}

In this section, we consider the bimodule species  $A = \begin{pmatrix} D_{0} & M \\ 0 & D_{1} \end{pmatrix}$ where $M$ is a $D_{0}-D_{1}$-bimodule with symmetric duals  and $\omega$ is the canonical complex $DA[-1]$ introduced in Lemma~\ref{lem.canonical}. As usual, we assume the left-right dimension $(m,n)$ of $M$ satisfies $mn \geq 4$. The main purpose of this section is to describe the indecomposable ``preprojective'' objects $(e_0A) \otimes_A^L \omega^{-i} , (e_1A) \otimes_A^L \omega^{-i} \ \  (i \in \mathbb{N})$ and indecomposable ``preinjective'' objects $(e_0A) \otimes_A^L \omega^i , (e_1A) \otimes_A^L \omega^i \ \  (i \in \mathbb{N})$ in terms of the noncommutative symmetric algebra $\mathbb{S} := \mathbb{S}^{nc}(M)$.  As in the classical theory, these will give analogues of the line bundles on $\mathbb{P}^1$. The preprojective objects were also essentially computed in \cite{dlab}, but their definition of preprojective objects is slightly different (ours is potentially a complex), and our calculation is also different, being an elegant direct computation based on the technology of Euler exact sequences in the theory of noncommutative symmetric algebras.

We compute $- \otimes^L_A\omega^{-1} = \operatorname{RHom}(DA,-)[1]$ by using the following bimodule right projective resolution of $DA$. It is a mild generalization of that constructed in \cite{king}.
\begin{equation} \label{eqn.king}
0  \rightarrow  (DA)e_{0} \otimes M \otimes e_{1}A  \rightarrow  ((DA)e_{0} \otimes e_{0}A) \oplus ((DA)e_{1} \otimes e_{1}A)  \rightarrow  DA  \rightarrow  0
\end{equation}
where the indicated tensor products are over appropriate $D_i$, and the maps are induced by multiplication.  This sequence equals
$$
0  \rightarrow  \begin{pmatrix} D_{0} \\ M^{*} \end{pmatrix} \otimes M \otimes e_{1}A  \rightarrow  \left(\begin{pmatrix} D_{0} \\ M^{*} \end{pmatrix} \otimes e_{0}A\right) \oplus \left(\begin{pmatrix} 0 \\ D_{1}\end{pmatrix} \otimes e_{1}A\right)  \rightarrow  DA  \rightarrow  0.
$$
Given a right $A$-module $P$, we wish to apply $\Hom_A(-,P)$ to the above resolution. The following lemma will assist us in this regard.
\begin{lemma} \label{lemma.adjoints}
Let $N$ be a finite-dimensional right $D_{i}$-module. Let $f_{1}, \ldots, f_{n}$ denote a right-basis for $N$, and $f_{1}^{*}, \ldots, f_{n}^{*}$ denote the dual left basis for ${N}^{*}$.  For any right $A$-module $P$, the function
\begin{equation} \label{eqn.adjoint}
\operatorname{Hom}_{A}(N \otimes e_{i}A, P) \rightarrow \operatorname{Hom}_{A}(e_{i}A, P) \otimes N^{*}
\end{equation}
defined by
$$
\psi \mapsto \sum_{j} \psi(f_{j} \otimes -)\otimes f_{j}^{*}
$$
is a group isomorphism natural in both $P$ and $N$. In particular, if $N$ is a $D_{i+1}-D_i$-bimodule, (\ref{eqn.adjoint}) is an isomorphism of right $D_{i+1}$-modules.
\end{lemma}

\begin{proof}
We describe a map (\ref{eqn.adjoint}) and leave it to the reader to check it is the map indicated.  We have isomorphisms
\begin{eqnarray*}
\operatorname{Hom}_{A}(N \otimes e_{i}A, P) & \overset{\cong}{\longrightarrow} & \operatorname{Hom}_{D_{i}}(N, \operatorname{Hom}_{A}(e_{i}A,P)) \\
& \overset{\cong}{\longrightarrow} & \operatorname{Hom}_{A}(e_{i}A, P) \otimes N^{*}
\end{eqnarray*}
by adjointness and the Eilenberg-Watts Theorem.
\end{proof}

Consider an $A-D_i$-bimodule $N = {N_0 \choose N_1}$ so $N_j$ is a $D_j-D_i$-bimodule and there is a multiplication map $\mu \colon M \otimes N_1 \rto N_0$. Taking the right dual of $\mu$ and using adjunction properties gives a new multiplication map $N_0^* \otimes M \rto N_1^*$ and hence an $A$-module structure on $(N_0^* \ \ N_1^*)$. We of course have

\begin{lemma}  \label{lem.formuladual}
There is an isomorphism of $D_i-A$-bimodules $N^* \cong (N_0^* \ \ N_1^*)$.
\end{lemma}

To be able to invoke the theory of noncommutative symmetric algebras,  we define the right $A$-modules
$$
P_{i} = \begin{cases} (\mathbb{S}_{-i0} \ \  \mathbb{S}_{-i1}) \mbox{ for $i \geq -1$} \\
({\mathbb{S}}^{*}_{0, -i-2}\ \  {\mathbb{S}}^{*}_{1, -i-2}) \mbox{ otherwise }\end{cases}
$$
with $A$-module multiplication induced by multiplication (or its dual) in the noncommutative symmetric algebra.

It follows from Lemmas \ref{lemma.adjoints} and \ref{lem.formuladual} and (\ref{eqn.king}), that $\operatorname{RHom}_{A}(DA,P_{i})$ is quasi-isomorphic to the complex
\begin{equation} \label{eqn.complex}
 P_ie_0 \otimes (D_{0} \quad M) \ \ \oplus \ \ P_ie_1 \otimes (0 \quad D_{1}) \xrightarrow{\phi} P_ie_1 \otimes M^{*} \otimes (D_{0} \quad M)
.\end{equation}
In order to explicitly compute $\operatorname{RHom}_{A}(DA,P_{i})$ (in Corollary \ref{cor:preproj}), we will need the Euler exact sequence, which we recall from \cite[Theorem 3.4 and Corollary 3.5]{abstractp1}.
\begin{theorem}  \label{thm.adam}
For $i \in \mathbb{Z}$, multiplication in $\mathbb{S}$ induces an exact sequence of right $\mathbb{S}$-modules
$$
0 \rightarrow Q_{i-2} \otimes \varepsilon_{i}\mathbb{S} \rightarrow \mathbb{S}_{i-2, i-1} \otimes \varepsilon_{i-1}\mathbb{S} \rightarrow \varepsilon_{i-2}\mathbb{S} \rightarrow \varepsilon_{i-2}\mathbb{S}/\varepsilon_{i-2}\mathbb{S}_{\geq i-1} \rightarrow 0.
$$
Furthermore, for all $i \leq j$, the canonical complex
$$
0 \rightarrow \mathbb{S}_{ij} \otimes Q_{j} \rightarrow \mathbb{S}_{i,j+1} \otimes M^{j+1*} \rightarrow \mathbb{S}_{i,j+2} \rightarrow 0
$$
is exact.
\end{theorem}

\begin{proposition} \label{prop.quasiiso}
If $i \geq -1$, then the map $\phi$ from (\ref{eqn.complex}) is injective and its cokernel is $P_{i+2}$.
\end{proposition}

\begin{proof}  We first check $\phi_0 = \phi \otimes_A Ae_0$ is injective with cokernel $P_{i+2}e_0$.  It suffices to prove that the adjoint of multiplication, $\mathbb{S}_{-i 0} \rightarrow \mathbb{S}_{-i 1} \otimes M^{*}$, is injective and has cokernel $\mathbb{S}_{-i-2, 0}$.  This follows from Theorem \ref{thm.adam} which gives exactness of
$$
0 \rightarrow \mathbb{S}_{-i 0} \otimes Q  \rightarrow \mathbb{S}_{-i 1} \otimes M^{*} \rightarrow \mathbb{S}_{-i 2} \rightarrow 0
.$$

We now examine $\phi_1 = \phi\otimes_A Ae_1$.  By definition of (\ref{eqn.complex}), the kernel of multiplication $\mathbb{S}_{-i 1} \otimes M^{*} \otimes M \rightarrow \mathbb{S}_{-i 3}$ contains $\im \phi_1$.

In addition, by Theorem \ref{thm.adam}, we have short-exact sequences
\begin{equation} \label{eqn.ses1}
0 \rightarrow \mathbb{S}_{-i 0} \otimes Q \otimes M \rightarrow \mathbb{S}_{-i 1} \otimes M^{*} \otimes M \rightarrow \mathbb{S}_{-i 2}\otimes M \rightarrow 0
\end{equation}
and
\begin{equation} \label{eqn.ses2}
0 \rightarrow \mathbb{S}_{-i 1} \otimes Q  \rightarrow \mathbb{S}_{-i 2} \otimes M \rightarrow \mathbb{S}_{-i 3} \rightarrow 0.
\end{equation}
The sequence (\ref{eqn.ses1}) gives an isomorphism
$$
\mathbb{S}_{-i 1} \otimes M^{*} \otimes M/\mathbb{S}_{-i 0} \otimes Q \otimes M \cong \mathbb{S}_{-i 2}\otimes M.
$$
Since the kernel of multiplication
$$
\mathbb{S}_{-i 2}\otimes M \rightarrow \mathbb{S}_{-i 3}
 \cong \mbbS_{-i-2,1}$$
is $\mathbb{S}_{-i 1} \otimes Q$ by (\ref{eqn.ses2}), $\phi_1$ is injective with cokernel $\mbbS_{-i-2,1} =P_{i+2}e_1$. Note that we have used the 2-periodicity of
$\mbbS$ above (see Remark~\ref{rem.2periodic}).


Thus, we conclude that the cokernel has the form $(\mathbb{S}_{-i-2, 0}\ \  \mathbb{S}_{-i-2, 1})$, and it is straightforward to show that the module structure on the cokernel agrees with $P_{i+2}$.
\end{proof}

\begin{prop}  \label{prop.preinj}
For $i \leq -4$, the map $\phi$ in (\ref{eqn.complex}) is injective and its cokernel is $P_{i+2}$.
\end{prop}

\begin{proof}
In this case our map $\phi$ is
$$ \mbbS^*_{0,-i-2} \otimes (D_0 \ \ M) \oplus \mbbS^*_{1,-i-2} \otimes (0 \ \ D_1) \xrightarrow{\phi}
 \mbbS^*_{1,-i-2} \otimes M^* \otimes (D_0 \ \ M) .$$
We first establish that $\phi \otimes_A Ae_0$ is injective with cokernel isomorphic to $P_{i+2} e_0$.  By Theorem \ref{thm.adam}, the sequence induced by multiplication
$$
0 \rightarrow Q \otimes \mbbS_{2, -i-2} \rightarrow M \otimes \mbbS_{1, -i-2} \xrightarrow{\pi} \mbbS_{0, -i-2} \rightarrow 0
$$
is exact. By naturality of the isomorphism in Lemma~\ref{lemma.adjoints}, $\phi\otimes_A Ae_0 = \pi^*$ which is injective with cokernel $\mbbS^*_{2, -i-2} \cong \mbbS^*_{0, -i-4} = P_{i+2}e_0$.

Now we analyze $\phi\otimes_A Ae_1$. Consider the commutative diagram
\begin{equation}
\begin{CD}
  @.  M^{*} \otimes Q \otimes \mbbS_{2, -i-2} @>{\nu_1}>> M^{*} \otimes M \otimes \mbbS_{1, -i-2}
@>{\phi_a}>> M^{*} \otimes \mbbS_{0, -i-2} \\
@. @V{\psi}V{\cong}V @VV{\phi_b}V @. \\
Q \otimes \mbbS_{3, -i-2} @>{\nu_2}>> M^{*} \otimes \mbbS_{2, -i-2} @>>> \mbbS_{1, -i-2} @.
\end{CD}
\label{eqn.dualmodule}
\end{equation}
whose rows are induced by multiplication and whose verticals are canonical. By Theorem \ref{thm.adam} again, the rows are short exact sequences (with zeros on the end omitted).

This time $\phi\otimes_A Ae_1 = (\phi_a^* \ \ \phi_b^*)$. Dualizing the above commutative diagram and using the fact that $\psi$ is an isomorphism shows that $\phi\otimes_A Ae_1$ is injective with cokernel isomorphic to $(Q \otimes \mbbS_{3, -i-2})^{*} \cong \mbbS_{3, -i-2}^{*} \cong \mbbS_{1, -i-4}^{*} = P_{i+2}e_1$.

To complete the proof, we must show that $\operatorname{coker} \phi$ is isomorphic to $P_{i+2}$ as $A$-modules. This amounts to showing that the following diagram is commutative
$$\begin{CD}
\mbbS^*_{1,-i-2} \otimes M^* \otimes M @>{\operatorname{coker} (\phi\otimes_A Ae_0) \otimes M}>> \mbbS^*_{2,-i-2} \otimes M\\
@| @VVV \\
\mbbS^*_{1,-i-2} \otimes M^* \otimes M @>{\operatorname{coker} (\phi\otimes_A Ae_1)}>> \mbbS^*_{3,-i-2}
\end{CD}$$
However, in the notation of diagram (\ref{eqn.dualmodule}), we see that $\operatorname{coker} (\phi\otimes_A Ae_0) \otimes M = \nu_1^*$ whilst $\operatorname{coker} (\phi\otimes_A Ae_1)$ is given by $\nu_2^* (\psi^*)^{-1}\nu_1^*$ so we are done.

\end{proof}

%

 We define a sequence $\mathcal{L}_{i}$ in the bounded derived category of right $A$-modules by
$$
\mathcal{L}_{i} = \begin{cases} P_{i} \mbox{ if $i \geq -1$} \\ P_{i}[-1] \mbox{ if $i < -1$} \end{cases}
$$

\begin{cor} \label{cor:preproj}
In $D_{fg}^{b}(A)$, we have an isomorphism $ \mathcal{L}_{i} \otimes^L_A \omega^{-1}   \cong \mathcal{L}_{i+2}$ for all $i \in \mathbb{Z}$.
\end{cor}
\begin{proof}
Propositions~\ref{prop.quasiiso} and \ref{prop.preinj} cover all cases except $i = -2,-3$ when we have
$$ \mathcal{L}_{i+2} \otimes_A^L \omega = e_{-i-2}A \otimes_A^L DA[-1]  =  e_{-i-2} DA[-1] = P_i [-1] = \mathcal{L}_i.$$
\end{proof}

\section{Beilinson equivalence and consequences}

In this section, we establish the main results of this paper, a version of Beilinson's derived equivalence, coherence of the noncommutative symmetric algebra and a version of Grothendieck's splitting theorem.

We will invoke (a mild generalization of) Polishchuk's theorem \cite[Proposition~2.3, Theorem~2.4]{polish} below. Let ${\sf C}$ be an abelian category and $\{L_i\}_{i \in \mathbb{Z}}$ a sequence of objects in ${\sf C}$ such that $D_i:= \End L_i$ is a right noetherian ring and $\Hom_{\sf C}(L_i, M)$ is a finitely generated $D_i$-module for every $M \in {\sf C}$.  We say that $\{L_i\}$ is {\em ample} if
\begin{itemize}
\item for every surjection $f \colon M \rto N$, the map $\Hom_{\sf C}(L_i, f)$ is surjective for $i \ll 0$ and,
\item for every $M \in {\sf C}, m \in \mathbb{Z}$, there exists a surjection of the form
$$ \oplus_{j = 1}^s L_{i_j} \lrto M $$
for some $i_j < m$.
\end{itemize}

\begin{theorem}  \label{thm.polish}
Let $\{L_i\}_{i \in \mathbb{Z}}$ be an ample sequence of objects in ${\sf C}$. Then the $\mathbb{Z}$-indexed algebra
$$ \mathbb{E} = \oplus_{i,j} \Hom_{\sf C}(L_{-j},L_{-i})$$
is coherent and ${\sf C} \equiv {\sf cohproj}\, \mathbb{E}$.
\end{theorem}
\noindent
\textbf{Remark:} The original statement in \cite{polish}, has more restrictive hypotheses, namely, Hom-finiteness.  However, Polishchuk in \cite[Remark~2 to Theorem~2.4]{polish} conceded a generalization like the one above should hold, and indeed one readily verifies that it holds with the same proof.

We need to invoke Minamoto's theory of Fano algebras \cite{minamoto2}. To this end, we consider an artinian ring $A$ of finite global dimension and let $\sigma \in D^b_{fg}(A)$ be a two-sided tilting complex. Minamoto defines the following full subcategories of $D^b_{fg}(A)$.

\begin{align*}
D^{\sigma,\geq 0} & = \{ M \in D^b_{fg}(A) | M \otimes^L_A \sigma^n\in D^{\geq 0}(A), \text{for all } n \gg 0\} \\
D^{\sigma,\leq 0} & = \{ M \in D^b_{fg}(A) | M  \otimes^L_A \sigma^n\in D^{\leq 0}(A), \text{for all } n \gg 0\}
\end{align*}

We thank the referee for correcting a missing hypothesis in the next result.
\begin{theorem}  \label{thm.minamoto}
Suppose that $\sigma^n$ is a pure $A$-module for all $n \gg 0$ and that $H^i(\sigma) = 0$ for $i >0$. If $A$ is hereditary, then the pair $(D^{\sigma,\leq 0}, D^{\sigma,\geq 0})$ defines a $t$-structure on $D^b_{fg}(A)$. Its heart ${\sf H}$ contains the objects $\{\sigma^n\}$ and  the sequence $\{\sigma^n\}$ is ample in {\sf H}. Furthermore, $D^b({\sf H})$ is triangle equivalent to $D^b_{fg}(A)$ and the global dimension of {\sf H} is at most one.
\end{theorem}
\begin{proof}
This is merely a combination of several of the main results of \cite[Section~3]{minamoto2}. The statements there include an additional assumption that $A$ is a finite-dimensional algebra over some field. However, this hypothesis is only used to ensure that the Hom-finiteness hypotheses in Polishchuk's theorem above hold. As we have seen, this is superfluous.

In detail, \cite[Theorem~3.15]{minamoto2} ensures that $(D^{\sigma,\leq 0}, D^{\sigma,\geq 0})$ defines a $t$-structure on $D^b_{fg}(A)$. By definition and purity of $\sigma^n$, the $\sigma^n \in {\sf H}$. Ampleness follows from \cite[Lemma~3.5]{minamoto2} whilst the triangle equivalence is \cite[Theorem~3.7(1)]{minamoto2}. Finally, the bound on the global dimension is given by \cite[Corollary~3.13]{minamoto2}.
\end{proof}

We now apply the theory above to noncommutative symmetric algebras. Let $A = \begin{pmatrix} D_{0} & M \\ 0 & D_{1} \end{pmatrix}$ as in Section~\ref{sec:species} where $M$ is a bimodule with symmetric duals and whose left-right dimension $(m,n)$ satisfies $mn \geq 4$. We saw that $A$ is artinian and hereditary. Let $\{\mathcal{L}_i\in D^b_{fg}(A)\}$ be the sequence defined in the paragraph preceding Corollary~\ref{cor:preproj}. Let $\mbbS = \mbbS^{nc}(M)$.

\begin{lemma}  \label{lem.endo}
Consider the $\mathbb{Z}$-indexed algebra
$$ \mathbb{E} := \oplus_{i,j} \Hom_{D^b_{fg}(A)}(\mathcal{L}_{-j},\mathcal{L}_{-i}) .$$
There is a natural isomorphism $\mbbS \cong \mathbb{E}$.
\end{lemma}
\begin{proof}
It suffices to show that we have compatible isomorphisms of  $\mathbb{Z}_{\leq l}$-indexed algebras
$$ \mbbS^{\leq l}:= \bigoplus_{i,j \leq l} \mbbS_{ij} \cong \bigoplus_{i,j\leq l} \Hom_{D^b_{fg}(A)}(\mathcal{L}_{-j},\mathcal{L}_{-i})=: \mathbb{E}^{\leq l} $$
for all $l$. Note first that
$$ \bigoplus_{i \leq 1} \mathcal{L}_{-i} = \left(\bigoplus_{j \geq -1} \mbbS_{-j0} \quad \bigoplus_{j \geq -1} \mbbS_{-j1}\right)$$
is naturally a $\mbbS^{\leq 1}-A$-bimodule so there is a natural algebra morphism
$$  \mbbS^{\leq 1} \lrto \bigoplus_{i,j\leq 1} \Hom_{D^b_{fg}(A)}(\mathcal{L}_{-j},\mathcal{L}_{-i}), $$
which we claim is an isomorphism.  Since this morphism sends $x \in \mathbb{S}_{ij}$ to left-multiplication by $x$, in order to prove the claim we must show every element of $\operatorname{Hom}_{A}(P_{-j},P_{-i})$ is induced by left-multiplication by a unique element of $\mathbb{S}_{ij}$. We first show that every element $\phi \in \operatorname{Hom}_{A}(P_{-j},P_{-i})$ extends uniquely to an element $\tilde{\phi} \in \operatorname{Hom}_{\mathbb{S}}((\varepsilon_{j}\mathbb{S})_{\geq 0}, \varepsilon_{i}\mathbb{S})$. To do so, we construct $\phi_n \colon \mathbb{S}_{jn} \rightarrow \mathbb{S}_{in}$ inductively, the case $n=0,1$ being the components of $\phi$. Consider the commutative diagram below, whose rows are exact by Theorem~\ref{thm.adam}.
\begin{equation*}
\begin{CD}
0 @>>> \mathbb{S}_{jn} \otimes Q_{n} @>>> \mathbb{S}_{j,n+1} \otimes M^{n+1*} @>>> \mathbb{S}_{j,n+2} @>>> 0  \\
@. @V{\phi_n}VV @V{\phi_{n+1}\otimes 1}VV @V{\phi_{n+2}}VV @. \\
0 @>>> \mathbb{S}_{in} \otimes Q_{n} @>>> \mathbb{S}_{i,n+1} \otimes M^{n+1*} @>>> \mathbb{S}_{i,n+2} @>>> 0
\end{CD}
\end{equation*}
Commutativity of the right-hand square defines $\phi_{n+2}$ given $\phi_n, \phi_{n+1}$ and furthermore, by construction, the resulting morphism $\tilde{\phi}$ is compatible with right mutliplication by $\mathbb{S}$.

Consider now the induced morphism 
$$
\Psi \colon \mathbb{S}_{ij} \simeq \operatorname{Hom}_{\mathbb{S}}(\varepsilon_{j}\mathbb{S}, \varepsilon_{i}\mathbb{S}) \rightarrow \operatorname{Hom}_{\mathbb{S}}((\varepsilon_{j}\mathbb{S})_{\geq 0}, \varepsilon_{i}\mathbb{S}).
$$ 
We know from \cite[Theorem 7.1 and Lemma 6.5]{abstractp1} that $\operatorname{Ext}^p_{\mathbb{S}}(\varepsilon_j\mathbb{S}/(\varepsilon_j\mathbb{S})_{\geq 0},\varepsilon_i \mathbb{S}) = 0$ for $p=0,1 $. The long exact sequence then shows that $\Psi$ is an isomorphism and the claim follows.

As noted in Remark~\ref{rem.2periodic}, the $\mathbb{Z}$-indexed algebra $\mbbS$ is 2-periodic whilst Corollary~\ref{cor:preproj} ensures that $\mathbb{E}$ is also 2-periodic, so by induction $\mbbS^{\leq l} \cong \mathbb{E}^{\leq l}$ for all $l$.
\end{proof}

\begin{theorem}  \label{thm.main}
Consider a $D_0-D_1$-bimodule $M$ with symmetric duals and whose left-right dimension $(m,n)$ satisfies $mn \geq 4$. Let $\mbbS = \mbbS^{nc}(M)$ be the corresponding noncommutative symmetric algebra.
\begin{enumerate}
\item The $\mathbb{Z}$-indexed algebra $\mbbS$ is coherent.
\item There is a triangle equivalence $D^b_{fg}({\sf cohproj}\,\mbbS) \cong D^b_{fg}(A)$ where the projective $\varepsilon_i \mbbS$ corresponds to $\mathcal{L}_i$.
\item The category ${\sf cohproj}\, \mbbS$ is hereditary.
\end{enumerate}
\end{theorem}
\begin{proof}
Note that $A \cong \mathcal{L}_{-1} \oplus \mathcal{L}_0$ so Corollary~\ref{cor:preproj} shows that $\omega^{-i} = \mathcal{L}_{2i-1} \oplus \mathcal{L}_{2i}$. For $i\geq 0$, this is always a pure module, so we may apply Theorem~\ref{thm.minamoto} to obtain an abelian subcategory {\sf H} of $D^b_{fg}(A)$, such that i) $\{\omega^{-i}\}$ is ample in {\sf H} ii) $D^b({\sf H}) \cong D^b_{fg}(A)$ and iii) {\sf H} has global dimension $\leq 1$.  The definition of ampleness immediately implies that $\{ \mathcal{L}_i\}$ is also an ample sequence in {\sf H} so Polishchuk's Theorem~\ref{thm.polish} together with Lemma \ref{lem.endo} yields part (1) and (2).  Part (3) now follows immediately from \ref{thm.minamoto}.
\end{proof}

The theory of coherent sheaves on $\mathbb{P} := {\sf cohproj}\, \mbbS$ can now easily be broached by examining the heart {\sf H} arising in the proof of Theorem~\ref{thm.main}. Note that {\sf H} contains the subcategory {\sf R} of regular modules defined in Section~\ref{sec:species}. Our point of view is that the corresponding subcategory {\sf T} of ${\sf cohproj}\,\mbbS$ are the  {\em torsion sheaves} on $\mathbb{P}$. Of course, the {\em torsion-free sheaves} corresponds to the additive subcategory {\sf F} generated by the $\varepsilon_i \mbbS$. The next result generalizes Grothendieck's splitting theorem and clarifies in what sense {\sf T} is like the subcategory of torsion coherent sheaves on $\mathbb{P}^1$.

\begin{corollary}  \label{cor.splitting}
With the above notation,
\begin{enumerate}
\item The indecomposable objects of ${\sf cohproj}\, \mbbS$ are the $\varepsilon_i \mbbS$ and the indecomposable objects of {\sf T}.
\item $({\sf T}, {\sf F})$ is a torsion pair in ${\sf cohproj}\, \mbbS$ i.e.
\begin{align*}{\sf T} & = {}^{\perp}{\sf F} := \{ \mathcal{N} \in {\sf cohproj}\, \mbbS | \Hom_{\mathbb{P}}(\mathcal{N}, {\sf F}) = 0\} \\
{\sf F} & = {\sf T}^{\perp} := \{ \mathcal{N} \in {\sf cohproj}\, \mbbS | \Hom_{\mathbb{P}}({\sf T}, \mathcal{N}) = 0\}
.\end{align*}
\item (Grothendieck splitting) In particular, {\sf F} is closed under extensions.
\item Every object in ${\sf cohproj}\, \mbbS$ is a direct sum of $\varepsilon_i \mbbS$ and its {\em torsion subsheaf}, that is, maximal subobject in {\sf T}.
\item Given an indecomposable $\mathcal{N} \in {\sf cohproj}\, \mbbS$, $\mathcal{N} \in {\sf T}$ if and only if the Hilbert function
$$
h_{\mathcal{N}} \colon i \mapsto \dim_{D_i}\operatorname{Hom}_{\mathbb{P}}(\varepsilon_{-i}\mathbb{S},\mathcal{N})-\dim_{D_i} \operatorname{Ext}^{1}_{\mathbb{P}}(\varepsilon_{-i}\mathbb{S},\mathcal{N})
$$
is non-negative.
\end{enumerate}
\end{corollary}
\begin{proof}
To prove parts~(1) and (2), it suffices to prove the analogous results about {\sf H}. Part~(1) follows from Lemma~\ref{lem.regular}~(1),(2). This together with Lemma~\ref{lem.regular}(3) gives part~(2). Part (3) follows from (2) and left exactness of $\operatorname{Hom}$.  Part (4) is now a standard result in torsion theory. Part~(5) follows from (2) and the classical Serre duality Theorem~\ref{thm.Serre} below.
\end{proof}
We remark here that wild behaviour means that {\sf T} is usually not closed under subobjects and the Hilbert functions of torsion sheaves are usually exponential.

\begin{theorem}  \label{thm.Serre}
For $\mathcal{M} \in {\sf cohproj}\, \mbbS$ and $p=0,1$, there is a natural isomorphism
$$
\operatorname{Ext}^{1-p}_{\mathbb{P}}(\varepsilon_{i} \mbbS,\mathcal{M}) \cong {}^{*}\operatorname{Ext}^{p}_{\mathbb{P}}(\mathcal{M}, \varepsilon_{i+2}\mbbS).
$$
\end{theorem}

\begin{proof}
The proofs in case $p=0$ and $p=1$ are similar.  In each case, one first notes that when $\mathcal{M} = \varepsilon_{j}\mbbS$, there exists an isomorphism, natural with respect to morphisms between objects of the form $\varepsilon_{l}\mbbS$, by \cite[Corollary 7.5]{abstractp1}.  One then proves the result for arbitrary $\mathcal{M}$ by using the fact that ${\sf cohproj}\, \mbbS$ is hereditary and $\mathcal{M}$ has a finite presentation.
\end{proof}

\end{document}